\documentclass[12pt]{article}
\usepackage[utf8]{inputenc}
\usepackage{amsmath,amssymb,amsthm}
\usepackage{mathrsfs}
\usepackage{graphicx}
\usepackage{url}
\usepackage{bm}
\usepackage{cite}
\usepackage{textcomp}
\allowdisplaybreaks
\usepackage{cases}

\usepackage{amsfonts}
\usepackage{graphicx}
\usepackage{makecell}
\usepackage{cite,color,xcolor}
\usepackage[margin=1.0 in]{geometry}
\usepackage[colorlinks,citecolor=blue,urlcolor=blue]{hyperref}
\usepackage[pagewise]{lineno}

\newcommand{\R}{\mathbb{R}}
\newcommand{\C}{\mathbb{C}}
\newcommand{\Z}{\mathbb{Z}}

\newcommand{\dd}{\mathrm{d}}

\newtheorem{theorem}{Theorem}
\newtheorem{lemma}{Lemma}
\newtheorem{definition}{Definition}
\newtheorem{remark}{Remark}

\title{Series decomposition of a class of special integrals}
\author{%
	Xiaolei Yang
	\\[6pt]
	\begin{minipage}{\textwidth}
		\centering
		\small  
		School of Mathematics and Information Science, Henan Polytechnic University, Jiaozuo 454003, China (Email: xlyang@hpu.edu.cn)
	\end{minipage}
}
\date{}

\begin{document}

\maketitle

\begin{abstract}
In this paper, we propose a new method for calculating integrals for a special class of integrands. As an application, we show how this method can be used to derive optimal pointwise temporal estimates for a class of nonlocal evolution equations. Compared with other methods, our approach can obtain both upper bound and lower bound simultaneously.
\end{abstract}

\section{Introduction}

  It is well known that integration is one of the most important mathematical tools, especially in the field of differential equations, where it is almost unavoidable. There have been extensive studies carried out in the framework of Sobolev spaces via the energy method, yielding many elegant and profound conclusions. However, such results usually reveal global properties rather than pointwise information. Here, for a special class of integrands, we propose a method to decompose the integral into a series. As an example, we show how it works on the nonlocal evolution equation with the form
\begin{equation}\label{1}
	\begin{cases}
		u_t(t,x) = d(J*u)(t,x) - u(t,x), & t > 0, x \in \mathbb{R} \\
		u(0,x) = u_0(x), & x \in \mathbb{R}
	\end{cases}
\end{equation}
where $d>0$ and \(\left( {J * u} \right)\left( {t,x} \right) = \int_{\mathbb{R}} {J\left( {x - y} \right)} u\left( {t,y} \right)\dd y.\) Throughout the text, assume that \(J(x)\) satisfies condition: 
\begin{flushleft}
	\textbf{(H1)} \(J(x)\in \C(\R)\) is a nonnegative function with \(\int_{\mathbb{R}} J(x) \dd x = 1\).
\end{flushleft} 

 As far as we know, Eq. \eqref{1} originates from Chasseigne, Chaves and Rossi\cite{Chasseigne2006}, where they discussed the asymptotic behavior of initial and boundary value problems for nonlocal diffusion equations $u_t = J*u - u.$ When different weights are assigned to the concentrations of substances at the same point for arrival and departure, we may get Eq. \eqref{1}. This scenario is quite common. For instance, many shopping websites pay more attention to the inflow of new customers than to the retention of existing ones.

 The method proposed here is based on the Leibniz test and Fourier transform. The remainder of the paper is organized as follows. In section 2, we presents the definition of the Fourier transform we adopt and its inversion. In section 3, we give the main results and the proof. In section 4, an example is given. 

\section{Preliminaries}

We recall some definitions and properties provided below which can be found in \cite{Grafakos2008}.

\begin{definition}
	For \( u \in L^1(\mathbb{R}) \), its Fourier transform \(\hat{u}=\mathscr{F}(u)\) and inverse Fourier transform \(\check{u}=\mathscr{F}^{-1}(u)\) are respectively defined as  
	\[
	\hat{u}(\xi) = \int_{-\infty}^{\infty} u(x) e^{-2\pi i x\xi} \, dx,  
	\quad  
	\check{u}(x) = \int_{-\infty}^{\infty} u(\xi) e^{2\pi i x\xi} \, d\xi.
	\]
	For \( u \in \mathscr{S}'(\mathbb{R}) \), the Fourier transform \(\hat{u}\) and inverse Fourier transform \(\check{u}\) are defined in the sense of distributions, that is, 
	\[\left\langle {\hat u,g}\right\rangle = \left\langle {u,\hat g} \right\rangle ,\quad\left\langle {\check{u} ,g} \right\rangle  = \left\langle {u,\check g} \right\rangle ,\quad \forall g\in \mathscr{S}(\mathbb{R}).\]
\end{definition}

\begin{lemma}[\cite{Grafakos2008}]\label{lem1}
	Given \(u\) in \(\mathscr{S}'(\mathbb{R}^n)\), we have \((\hat{u})\check{ }=u\).
\end{lemma}

\begin{lemma}\label{lem2}
	For any $\varepsilon>0$, for any $\xi_0\neq0$, there exists an even continuous probability density $J(x)$ such that \(\hat{J}(\xi)\) is radially decreasing and satisfies
	\begin{equation*}
		\hat{J}(\xi_0)>1-\varepsilon.
	\end{equation*}
\end{lemma}
\begin{proof}
	Consider the centered normal density \( J_\sigma (x) = \frac{1}{\sqrt{2\pi}\sigma} e^{-\frac{x^2}{2\sigma^2}}\) with \(\sigma>0. \)
\end{proof}

\section{Main results}

\begin{theorem}\label{thm1}
 Let $u_{0}(x)\in \mathscr{S}'(\mathbb{R})$ and $\hat{u}_{0}(x)\in L^{1}(\R)$. Under the condition of \textbf{(H1)}, denote the solution of Eq.\eqref{1} as \(u(t,x)\), then,
\begin{equation}\label{T1}
	|u(t,x)|\leq \|\hat{u}_0\|_{L^1(\R)}e^{(d-1)t}.
\end{equation}
Furthermore, $e^{(d-1)t}$ is sharp in the sense that for any $\varepsilon>0$, there are $J(x)$ and $  u_0(x)$  such that $u(t,x)$ satisies 
\begin{equation}\label{T2}
		g(x)e^{(d-1-\varepsilon)t}<|u(t,x)|<h(x)e^{(d-1)t},
\end{equation}
 where $g(x), h(x)>0$ are determined only by $u_{0}(x)$. 
\end{theorem}

\begin{proof}
When \( J(x) = \delta(x) \), noting that \( (\delta * u)(t,x) = u(t,x) \), system \eqref{1} then reduces to the local case
\begin{equation*}
\begin{cases}
u_t(t,x) = (d-1)u(t,x), & t > 0, x \in \mathbb{R}, \\
u(0,x) = u_0(x), & x \in \mathbb{R},
\end{cases}
\end{equation*}
which has the solution \( u(t,x) = e^{(d-1)t}u_0(x) \). The rate \(e^{(d-1)t}\) emerges here.

{\bf Step 1:} we show that \(e^{(d-1)t}\) is an upper bound. Let $u_{0}(x)\in \mathscr{S}'(\mathbb{R})$ and $\hat{u}_{0}(x)\in L^{1}(\R)$. Applying the Fourier transform to \eqref{1} gives
\(\hat{u}(t,\xi) = e^{t(d\hat{J}(\xi)-1)}\hat{u_0}(\xi)\in L^1(\R)\). The result follows from the fact that $\parallel \hat{f}\parallel_{L^{\infty}(\mathbb{R})}\leq\parallel f\parallel_{L^{1}(\mathbb{R})}$ for $f\in L^{1}(\R)$.

{\bf Step 2:} We choose suitable functions to ensure the sharpness of \textbf{Theorem \ref{thm1}}. Here we impose the following condition:\\
\textbf{(H2)} \(\hat{u}_0\) is a nonagetive even function and radially decreasing, whatsmore, for any \(x>0\),
\begin{equation}\label{h2}
	\int_0^{\frac{3}{4x}}\hat{u}_0(\xi)\cos(2\pi x\cdot\xi)\dd\xi>0.
\end{equation}

From now on, let \(J\) is a centered normal distribution whose parameter \(\sigma\) to be determined later. Then,
\begin{align*}
	u(t,x) &= \mathscr{F}^{-1}\left(\hat{u}(t,\xi)\right) = \int_{\mathbb{R}} e^{t(d\hat{J}(\xi)-1)}\hat{u}_0(\xi)e^{2\pi i x \xi}\dd \xi \\
	&= \int_{\mathbb{R}} e^{t(d\hat{J}(\xi)-1)}\hat{u}_0(\xi)(\cos 2\pi x \xi + i \sin 2\pi x \xi)\dd \xi \\
	&= \int_{\mathbb{R}} e^{t(d\hat{J}(\xi)-1)}\hat{u}_0(\xi)\cos(2\pi x \xi)\dd \xi.
\end{align*}
where we have used the facts that \(\sin (2\pi x \xi)\) is an odd function and \(e^{t(d\hat{J}(\xi)-1)}\hat{u}_0(\xi)\in L^1(\R)\). It is easy to know that \(u(t,x)\) is even, then, it suffices to consider the case \(x \geq 0\). 

 When \(x > 0\), let
\[
A_k = \int_{\frac{2k-1}{4x}}^{\frac{2k+1}{4x}} \hat{u}_0 (\xi) \cos (2\pi x \xi) \dd \xi, \quad k \in \Z;
\]
\[
B_k = \int_{\frac{2k-1}{4x}}^{\frac{2k+1}{4x}} e^{t(d\hat{J}(\xi)-1)} \hat{u}_0 (\xi) \cos (2\pi x \xi) \dd \xi, \quad k \in \mathbb{Z}.
\]
Here, \(A_k\) and \(B_k\) are functions of \(x\), and the same applies to \(C_k\) mentioned later. Note that \(\hat{J}(\xi)=e^{-2\pi^2 \sigma^2\xi^2}\) is a radially decreasing positive function, and for every \(k\), the integration intervals for both \(A_k\) and \(B_k\) are \(\left( (\frac{2k-1}{4x}),(\frac{2k+1}{4x})\right) \), then \(2\pi x\xi\in (k\pi-\frac{\pi}{2},k\pi+\frac{\pi}{2})\), thus integrands associated with \(A_k\) and \(B_k\) do not change sign in the corresponding interval. With these properties, we can show that:

(i) \( A_k = A_{-k}, B_k = B_{-k} \); Moreover,
\[
A_k, B_k
\begin{cases} 
	> 0, & k = 2m;\\ 
	< 0, & k = 2m + 1.
\end{cases}
\]

(ii) \(B_k\) can be estimated by \(A_k\),
\begin{equation}\label{5}
	|A_k|e^{t\left( d\hat{J}\left( \frac{2k+1}{4x} \right)-1 \right)} < |B_k| < |A_k|e^{t\left( d\hat{J}\left( \frac{2k-1}{4x} \right)-1 \right)}, \quad |k| \geq 1
\end{equation}
\begin{equation}\label{6}
	|A_0|e^{t\left( d\hat{J}\left( \frac{1}{4x} \right)-1 \right)} < |B_0| < |A_0|e^{t(d-1)}.
\end{equation}

(iii) \(|A_k|\) and \(|B_k|\) are both monotonically decreasing for \(k\geq0\). Indeed, when \(k\geq 1\),
\begin{align*}
	|A_{k+1}| &= \left| \int_{\frac{2k+1}{4x}}^{\frac{2k+3}{4x}} \hat{u}_0 (\xi) \cos (2\pi x \xi) \dd \xi\right|=\int_{\frac{2k+1}{4x}}^{\frac{2k+3}{4x}} \left| \hat{u}_0 (\xi) \cos (2\pi x \xi)\right|  \dd \xi  \\
	&=  \int_{\frac{2k-1}{4x}}^{\frac{2k+1}{4x}} \left|\hat{u}_0 \left(\xi+\frac{1}{2x}\right) \cos (2\pi x \xi + \pi)\right| \dd \xi= \int_{\frac{2k-1}{4x}}^{\frac{2k+1}{4x}} \left| \hat{u}_0 \left(\xi+\frac{1}{2x}\right) \cos (2\pi x \xi)\right|  \dd \xi  \\
	&< \int_{\frac{2k-1}{4x}}^{\frac{2k+1}{4x}} \left| \hat{u}_0 \left(\xi\right) \cos (2\pi x \xi)\right|\dd \xi=|A_k|.
\end{align*}
Besides, \(|A_0|>|A_1|\) also holds since \(|\xi+\frac{1}{2x}|<|\xi|\) for \(\xi\in (-\frac{1}{4x},\frac{1}{4x})\). Taking into consideration that \( e^{t(d\hat{J}(\xi)-1)}\hat{u}_0(\xi)\) also decreases radially, we obtain the monotonicity of \(|B_k|\) for \(k\geq0\).

(iv) Since \(\hat{u}_0, e^{t(d\hat{J}-1)} \hat{u}_0\in L^1(\R)\), then, series \(A_n\) and \(B_n\) are absolutely convergent. It follows that \(A_k, B_k\rightarrow0\) as \(k\rightarrow \infty\).

Now, define sequence \(\{C_n\}_{n=0}^\infty\) as
\[C_n = 
\begin{cases} 
	B_0, & n = 0, \\ 
	B_n+B_{-n}, & n > 0. 
\end{cases}\]
Then, we have
\begin{align*}
u(t,x) = \int_{\mathbb{R}} e^{t(d\hat{J}(\xi)-1)}\hat{u}_0(\xi)\cos(2\pi x \xi)\dd \xi
= \sum_{k=-\infty}^{\infty} B_k=\sum_{k=0}^{\infty} C_k.
\end{align*}
 Combined with E.q. \eqref{5}, \eqref{6}, we know that 
 \[0<A_0e^{t\left( d\hat{J}\left( \frac{1}{4x} \right)-1 \right)} < C_0=B_0 < A_0e^{t(d-1)},\]
 	\[A_1 e^{t\left( d\hat{J}\left( \frac{1}{4x} \right)-1 \right)} < B_1=B_{-1} < A_1 e^{t\left( d\hat{J}\left( \frac{3}{4x} \right)-1 \right)}, \]
 Thus \(C_1=B_1+B_{-1}\geq2A_1 e^{t\left( d\hat{J}\left( \frac{1}{4x} \right)-1 \right)}\). Note that condition \textbf{(H2)} means \(\frac{1}{2}A_0+A_1>0\), hence \(\{C_n\}_{n=0}^{\infty}\) is also an alternating series that Leibniz test can be applied. What's more,
  \begin{equation}\label{7}
  	(A_0 + 2A_1)e^{t(d\hat{J}(\frac{1}{4x})-1)} <C_0+C_1< u(t,x) <C_0< A_0e^{t(d-1)}.
  \end{equation}
  
 Fix \(x_0>0\) small enough such that
 \begin{equation*}
 	\int_{-\frac{1}{8x_0}}^{\frac{1}{8x_0}}\hat{u}_0(\xi)\dd \xi\geq \frac{3\|\hat{u}_0\|_{L^1(\R)}}{4},
 \end{equation*}
 Then, for \(0\leq x<x_0 \),
 \begin{align*}
 	u(t,x) &= \int_{\mathbb{R}} e^{t(d\hat{J}(\xi)-1)}\hat{u}_0(\xi)\cos(2\pi x \xi)\dd \xi \\
 	&= \int_{|\xi|\leq \frac{1}{8x_0}} e^{t(d\hat{J}(\xi)-1)}\hat{u}_0(\xi)\cos(2\pi x \xi)\dd \xi+ \int_{|\xi|> \frac{1}{8x_0}} e^{t(d\hat{J}(\xi)-1)}\hat{u}_0(\xi)\cos(2\pi x \xi)\dd \xi\\
 	&\geq \int_{|\xi|\leq \frac{1}{8x_0}} e^{t(d\hat{J}(\xi)-1)}\hat{u}_0(\xi)\cos(2\pi x \xi)\dd \xi-\int_{|\xi|> \frac{1}{8x_0}}\left|  e^{t(d\hat{J}(\xi)-1)}\hat{u}_0(\xi)\cos(2\pi x \xi)\right| \dd \xi\\
 	&\geq \frac{\sqrt{2}}{2}\int_{|\xi|\leq \frac{1}{8x_0}} e^{t(d\hat{J}(\xi)-1)}\hat{u}_0(\xi)\dd \xi-\int_{|\xi|> \frac{1}{8x_0}}\left|  e^{t(d\hat{J}(\xi)-1)}\hat{u}_0(\xi)\cos(2\pi x \xi)\right| \dd \xi\\
 	&\geq \frac{\sqrt{2}}{2}e^{t(d\hat{J}(\frac{1}{8x_0})-1)}\int_{|\xi|\leq \frac{1}{8x_0}} \hat{u}_0(\xi)\dd \xi-e^{t(d\hat{J}(\frac{1}{8x_0})-1)}\int_{|\xi|> \frac{1}{8x_0}}\left| \hat{u}_0(\xi)\cos(2\pi x \xi)\right| \dd \xi\\
 	&\geq \frac{3\sqrt{2}}{8}e^{t(d\hat{J}(\frac{1}{8x_0})-1)} \|\hat{u}_0\|_{L^1(\R)}-e^{t(d\hat{J}(\frac{1}{8x_0})-1)}\int_{|\xi|> \frac{1}{8x_0}}\left| \hat{u}_0(\xi)\right| \dd \xi\\
 	&\geq \left( \frac{3\sqrt{2}}{8}-\frac{1}{4}\right)\|\hat{u}_0\|_{L^1(\R)}e^{t(d\hat{J}(\frac{1}{8x_0})-1)} .
 \end{align*}
 
 \eqref{7} shows that for \(x\geq x_0\), 
 \[u(t,x)\geq (A_0 + 2A_1)e^{t(d\hat{J}(\frac{1}{4x_0})-1)}.\] 
 In summary, we have
 \[u(t,x)\geq c(x)e^{t(d\hat{J}(\frac{1}{4x_0})-1)}\]
 where
 \[c(x)  :=  \left\{ \begin{array}{l}
 	\frac{{3\sqrt 2-2 }}{8},\quad 0 \le x < {x_0},\\
 	{A_0} + 2{A_1},\quad x > {x_0}.
 \end{array} \right.\]
Remember that \(\hat{J}(\xi)=e^{-2\pi^2 \sigma^2\xi^2}\), then, for any \(\varepsilon > 0\), ther exists \(\sigma > 0\) such that \(\hat{J}\left(\frac{1}{4x}\right) > 1 - \frac{\varepsilon}{d}\), then
 \[e^{t(d\hat{J}(\frac{1}{4x})-1)} >e^{t(d-\varepsilon-1)},\]
  that is, \(e^{t(d-1)}\) is sharp.
\end{proof} 

\section{Examples}

We claim that there exist functions satisfying \textbf{(H2)} in \(L^1(\R)\). Indeed, \(\frac{e^{-|\xi|}}{\sqrt{\xi}}\) is such a function. A direct calculation gives
\[\int_0^{\frac{3}{{4x}}} {\frac{{\cos \left( {2\pi x \cdot \xi } \right)}}{{\sqrt \xi  }}\dd\xi  = \frac{1}{{\sqrt x }}} \mathrm{FresnelC}\left( {\sqrt 3 } \right) > 0,\]
where \(\mathrm{FresnelC}(x)\) is the Fresnel cosine integral defined by
\[\mathrm{FresnelC}(x)= \int_0^x {\cos \left( {\frac{\pi }{2}{t^2}} \right)} dt.\]
Thus, we have
\[\int_0^{\frac{1}{{4x}}} {\frac{{{e^{ - \xi }}}}{{\sqrt \xi  }}} \cos \left( {2\pi x \cdot \xi } \right)d\xi  > {e^{ - \frac{1}{{4x}}}}\mathrm{FresnelC}(1)>0,\]
\[ {e^{ - \frac{1}{{4x}}}}\left( {\mathrm{FresnelC}(\sqrt 3 ) - \mathrm{FresnelC}(1)} \right)<\int_{\frac{1}{{4x}}}^{\frac{3}{{4x}}} {\frac{{{e^{ - \xi }}}}{{\sqrt \xi  }}} \cos \left( {2\pi x \cdot \xi } \right)\dd\xi<0 .\] 
Then, we obtain
\[\int_0^{\frac{3}{{4x}}} {e^{ - \xi }\frac{{\cos \left( {2\pi x \cdot \xi } \right)}}{{\sqrt \xi  }}\dd\xi>\frac{1}{{e^{\frac{1}{{4x}}}\sqrt x }}} \mathrm{FresnelC}\left( {\sqrt 3 } \right) > 0.\]
Besides, \(\frac{e^{-|\xi|}}{\sqrt{|\xi|}}\in L^1(\R)\) is obvious.

We now turn our attention to \(\mathscr{F}^{-1}(\hat{u}_0)\):
\begin{align*}
	\mathscr{F}^{-1}(\hat{u}_0)&=\int_\R {\frac{1}{{{e^{  \left| \xi  \right|}}\sqrt {\left| \xi  \right|} }}} {e^{2\pi ix \cdot \xi }}\dd\xi  = 2\int_0^\infty   {\frac{1}{{{e^\xi}\sqrt {\xi} }}} \cos(2\pi x\cdot \xi)\dd\xi \\
	 & =2\mathrm{Re}\int_0^\infty \xi^{-\frac{1}{2}} e^{-(1+2\pi xi)\xi}\dd\xi=2\mathrm{Re}\left((1+2\pi x i)^{-\frac{1}{2}}\Gamma\left( \frac{1}{2}\right)  \right)\\
	 &=2\sqrt{\pi} \mathrm{Re}(1+2\pi x i)^{-\frac{1}{2}}
\end{align*}
Denote \(1+2\pi x i=\sqrt{1+4\pi^2x^2}e^{i\theta}\), where \(\sin\theta=\frac{1}{\sqrt{1+4\pi^2 x^2}},  \cos\theta=\frac{2\pi x}{\sqrt{1+4\pi^2x^2}}\). Then,
\[(1+2\pi \xi i)^{-\frac{1}{2}}=(1+4\pi^2x^2)^{-\frac{1}{4}}e^{-\frac{i\theta}{2}.}\]
Thus, \(\mathscr{F}^{-1}(\hat{u}_0)(x)=2\sqrt{\pi}(1+4\pi^2x^2)^{-\frac{1}{4}}\cos(\frac{\theta}{2})=\sqrt{2\pi}\sqrt{\frac{\sqrt{1+4\pi^2x^2}+1}{1+4\pi^2x^2}}\in \mathscr{S}'(\R).\) By \textbf{Lemma \ref{lem1}}, \(u_0=\mathscr{F}^{-1}(\hat{u}_0)\) in \(\mathscr{S}'(\R).\)

\begin{remark}
	Indeed, \(\frac{f(\xi)}{\sqrt{\xi}}\) is suitable whenever \(f\in L^1(\R)\) is a positive function and decreases radially.
\end{remark}

\begin{remark}	
	We can consider other series such as geometric series. 
\end{remark}

\section{Acknowledgements}

The author would like to express gratitude to Professor Yanghai Yu for his series of suggestions on the presentation of the article. X. Yang is supported by the Natural Science Foundation of Henan Province (Grant No. 242300421682), the Fundamental Research Funds for the Universities of Henan Province (Grant No. NSFRF2502092), the Doctor’s funding of Henan Polytechnic University (Grant No. B2021-56).

\end{document}